\documentclass{article}

   \usepackage[preprint,nonatbib]{neurips_2020}

\usepackage[utf8]{inputenc} 
\usepackage[T1]{fontenc}    
\usepackage{hyperref}       
\usepackage{url}            
\usepackage{booktabs}       
\usepackage{amsfonts}       
\usepackage{nicefrac}       
\usepackage{microtype}      

\usepackage{algorithm,algorithmic}

\usepackage{amsmath, amssymb}
\usepackage{graphicx}
\usepackage{subcaption}
\usepackage{accents}
\usepackage{hyperref}

\usepackage{xcolor, colortbl}
\usepackage{wrapfig,lipsum,booktabs}
\usepackage{amsthm}
\usepackage{mathtools}
\mathtoolsset{showonlyrefs}

\newcommand{\R}{{\mathbb R}} 
\def\R{\mathbb R}
\newcommand{\E}{{\mathbb E}}

\newcommand{\Prob}{{\mathbb P}}

\newcommand{\eps}{\varepsilon}
\newcommand{\la}{\langle}
\newcommand{\ra}{\rangle}

\def\E{\mathbb E}

\newcommand{\norm}[1]{\left\lVert#1\right\rVert_2}
\newcommand{\leqarg}[1]{\ensuremath{\stackrel{\text{#1}}{\leq}}}
\newcommand{\eqarg}[1]{\ensuremath{\stackrel{\text{#1}}{=}}}
\usepackage{makecell}
\usepackage{tikz}
\newcommand*\circled[1]{\tikz[baseline=(char.base)]{\node[shape=circle,draw,inner sep=2pt] (char) {#1};}}

\newtheorem{theorem}{Theorem}
\newtheorem{lemma}{Lemma}

\usepackage{xcolor}

\def\pd#1{{#1}} 
\def\dd#1{{#1}} 
\def\at#1{{#1}} 
\def\att#1{{#1}}

\title{Adaptive Gradient Descent for Convex and Non-Convex Stochastic Optimization}

\author{%
  Darina Dvinskikh\\
  Weierstrass Institute for Applied Analysis and Stochastics,\\
  Institute for Information Transmission Problems RAS\\
  \texttt{darina.dvinskikh@wias-berlin.de} \\
  \And
  Aleksandr Ogaltsov\\
  Higher school of economics \\
  \texttt{ogalcov-aleksand@mail.ru} \\
  \AND
  Alexander Gasnikov\\
  Institute for Information Transmission Problems RAS,\\
    Higher school of economics, \\
  Moscow Institute of Physics and Technology\\
  \texttt{gasnikov@yandex.ru} \\
  \And
  Pavel Dvurechensky\\
   Weierstrass Institute for Applied Analysis and Stochastics,\\
  Institute for Information Transmission Problems RAS\\
  \texttt{pavel.dvurechensky@wias-berlin.de} \\
  \And
  Alexander Tyurin\\
  Higher school of economics \\
   \texttt{alexandertiurin@gmail.com}
  \And
  Vladimir Spokoiny\\
  Weierstrass Institute for Applied Analysis and Stochastics\\
  \texttt{spokoiny@wias-berlin.de} \\
}

\begin{document}

\maketitle

\begin{abstract}
In this paper we propose several adaptive gradient methods for stochastic optimization. Unlike AdaGrad-type of methods, our algorithms are based on Armijo-type line search and they simultaneously adapt to the unknown Lipschitz constant of the gradient and variance of the stochastic approximation for the gradient. We consider an accelerated and non-accelerated gradient descent for convex problems and gradient descent for non-convex problems. In the experiments we demonstrate superiority of our methods to existing adaptive methods, e.g. AdaGrad and Adam.
\end{abstract}

\section{Introduction}
In this paper we consider unconstrained minimization problem
\begin{equation}
    \label{eq:pr_st}
    \min_{x \in \R^n} f(x), 
\end{equation}
where $f(x)$ is a smooth, possibly non-convex function with $L$-Lipschitz continuous gradient. 
We say that a function $f: E \to \R$ has  $L$-Lipschitz continuous gradient  if it is continuously differentiable and its gradient satisfies
\begin{equation*}
f(y) \leq f(x) + \la \nabla f(x) , y-x \ra + \frac{L}{2} \|x-y\|^2_2, \quad \forall ~x,y \in E.
\end{equation*}
We assume that the access to the objective $f$ is given through stochastic oracle $\nabla f(x,\xi)$, where $\xi$ is a random variable. The main assumptions on the stochastic oracle are standard for stochastic approximation literature \cite{nemirovski2009robust}
\begin{equation}
     \label{eq:stoch_or_asmpt2}
    \E \nabla f(x,\xi) = \nabla f(x), \;\; \E \left(\|\nabla f(x,\xi)-\nabla f(x)\|_2^2 \right) \leq D.
\end{equation}
One of the cornerstone questions for optimization methods is the choice of the stepsize, which has a dramatic impact on the convergence of the algorithm and the quality of the output, as, e.g., in deep learning, where it is called learning rate. Standard choice of the stepsize for the gradient descent in deterministic optimization is $1/L$ \cite{nesterov2018lectures} and it is possible to use the (accelerated) gradient descent without knowing this constant \cite{nesterov2013gradient,malitsky2018first-order,dvurechensky2018computational} in convex case and gradient method  \cite{bogolubsky2016learning} 
in non-convex case using an Armijo-type line search and checking whether the quadratic upper bound based on the $L$-smoothness is correct. Another option to adapt to the unknown smoothness is to use small-dimensional relaxation \cite{nemirovski1982orth,nesterov2018primal-dual}
. So far there is only partial understanding of how to generalize these ideas for stochastic optimization. Heuristic adaptive algorithms for smooth strongly convex optimization is proposed in \cite{byrd2012sample,friedlander2012hybrid} (see also review \cite{newton2018stochastic}). Theoretical analysis in these papers is made for the idealised versions of their algorithms which either are not practical or not adaptive. 
Authors of  work \cite{iusem2019variance} propose and theoretically analyse a method for stochastic monotone variational inequalities.

Another way to construct an adaptive stepsize comes from non-smooth optimization \cite{polyak1987introduction}, where it is suggested to take it as $1/\|\nabla f(x)\|_2$. This idea turned out to be very productive and allowed to introduce stochastic adaptive methods \cite{duchi2011adaptive,adam,deng2018optimal}
, among which usually the Adam algorithm is a method of choice \cite{ruder2016overview}. Recently this idea was generalized to propose adaptive methods for 
smooth stochastic convex optimization, yet with acceleration only for non-stochastic optimization in \cite{levy2018online}; for 
smooth stochastic monotone variational inequalities and convex optimization problems in \cite{bach2019universal,kavis2019unixgrad}; for smooth non-convex stochastic optimization \cite{ward19adagrad}. 
One of the main drawbacks in these methods is that they are not flexible to mini-batching approach, which is widespread in machine learning. The problem is that to choose the optimal mini-batch size (see also \cite{gazagnadou2019optimal}) for all these methods, one needs to know all the parameters like $D$, $L$ and the adaptivity vanishes. Moreover, these methods usually either need some additional information about the problem, e.g., distance to the solution of the problem, which may not be known for a particular problem, or have a set of hyperparameters, the best values for which are not readily available even in the non-adaptive setting. Finally, the stepsize in these methods is decreasing and in the best case asymptotically converges to a constant stepsize of the order $1/L$. This means that the methods could not adapt to the local curvature of the objective function and converge faster in the areas where the function is smoother. We summarize available literature in the Table above.

In this paper we follow an alternative line, trying to extend the idea of Armijo-type line search for the adaptive methods for convex and non-convex stochastic optimization. Surprisingly, the adaptation is needed not to each parameter separately, but to the ratio $D/L$, which can be considered as signal to noise ratio or an effective Lipschitz constant of the gradient in this case. We propose an accelerated and non-accelerated gradient descent for stochastic convex optimization and a gradient method for stochastic non-convex optimization. Our methods are flexible enough to use optimal choice of mini-batch size without additional information on the problem. Moreover, our procedure allows an increase of the stepsize, which accelerated the methods in the areas where the Lipschitz constant is small. Also, as opposed to the existing methods, our algorithms do not need to know neither the distance  to the solution, nor a set of complicated hyperparameters, which are usually fine-tuned by multiple repetition of minimization process. Moreover, since our methods are based on inexact oracle model (see e.g., \cite{devolder2014first,gasnikov2016stochasticInter}), they are adaptive not only for a stochastic error, but also for deterministic, e.g., non-smoothness of the problem. This means that our methods are universal for smooth and non-smooth optimization \cite{nesterov2015universal,yurtsever2015universal}. 
Finally, we demonstrate in the experiments that our methods work faster than state-of-the-art methods \cite{duchi2011adaptive,adam}.
{\footnotesize
\begin{table}[t]
\setlength{\tabcolsep}{3pt}
\begin{center}
\begin{small}
\begin{sc}
\begin{tabular}{lcccccc}
\toprule
Paper & N-C\footnotemark
& N-Ac & Ac & Prf & Btch & Par \\
\midrule
Duchi et al.'11 & $\times$ & $\surd$ & $\times$ & $\surd$ & $\times$ & $\times$ \\
Byrd et al.'12 & $\times$ & $\surd$ & $\times$ & $\times$ & $\times$ & $\times$ \\
Friedlander et al.'12 & $\times$ & $\surd$ & $\times$ & $\times$ & $\times$ & $\times$ \\
Kingma \& Ba'15 & $\times$ & $\surd$ & $\surd$ & $\times$ & $\times$ & $\times$ \\
Iusem et al.'19 & $\times$ & $\surd$ & $\times$ & $\surd$ & $\times$ & $\surd$\\
Levy et al.'18   & $\times$ & $\surd$ & $\times$ & $\surd$ & $\times$ & $\times$\\
Deng et al.'18   & $\times$ & $\times$ & $\surd$ & $\times$ & $\times$ & $\times$\\
Ward et al.'19        & $\surd$ & $\times$ & $\times$ & $\surd$ & $\times$ & $\surd$\\
Bach \& Levy'19   & $\times$ & $\surd$ & $\times$ & $\surd$ & $\times$ & $\times$\\
\textbf{This paper}  & $\surd$  & $\surd$ & $\surd$  & $\surd$ & $\surd$ & $\surd$ \\
\bottomrule
\end{tabular}
\end{sc}
\end{small}
\end{center}
\end{table}
\footnotetext{N-C stands for availability of an algorithm for non-convex optimization, N-Ac for a non-accelerated algorithm for convex optimization, Ac for accelerated algorithm for convex optimization, Prf for proof of the convergence rate, Btch for possibility to adaptively choose batch size without knowing other parameters, Par for non-necessity to know or tune hyperparameters like distance to the solution for choosing the stepsize.}
}

The paper is structured as follows. In Sect. \ref{Sect:conv} we present two stochastic algorithms based on stochastic gradient method to solve optimization problem of type \eqref{eq:pr_st} with convex objective function. The first algorithm is accompanied by the complexity bounds on  total number of iterations and oracle calls for the approximated stochastic gradients. The second algorithm is fully-adaptive and does not require the knowledge of Lipschitz constant for the gradient of the objective and the variance of its stochastic approximation.
Sect. \ref{Sect:non_conv} renews Sect. \ref{Sect:conv} for non-convex objective function. Finally, in Sect. \ref{Sect:exper} we show numerical experiments supporting the theory in above sections.

\section{Stochastic convex optimization}\label{Sect:conv}


In this  section we solve problem \eqref{eq:pr_st} for convex objective. 
Assuming the Lipschitz constant for the continuity of the objective gradient to be known we prove the complexity bounds for the proposed algorithm. Then we refuse this assumption and provide complexity bounds for the adaptive version of the algorithm which does not need the information about Lipschitz constant.


\subsection{Non-adaptive algorithm}
Firstly, we consider stochastic gradient descent with general stepsize $h$ 
\begin{equation}\label{eq:grad_meth}
    x^{k+1} = x^k - h\nabla^{r} f(x^k, \{\xi^{k+1}_l\}_{l=1}^{r}).
\end{equation}
where $ \nabla^{r} f(x,\{\xi_l\}_{l=1}^{r})$ is  stochastic approximation for the gradient $\nabla f(x)$ with mini-batch of size $r$ 
\begin{equation}\label{eq_minibatch}
    \nabla^r f(x,\{\xi_l\}_{l=1}^r) = \frac{1}{r}\sum_{l=1}^r  \nabla f(x,\xi_l).
\end{equation}
Here each stochastic gradient $\nabla f(x,\xi_l)$
satisfies \eqref{eq:stoch_or_asmpt2}. 

We start with the Lemma characterizing the decrease of the objective on one step of the algorithm \eqref{eq:grad_meth}.
\begin{lemma}\label{Lm_step_size}
For the stochastic gradient descent \eqref{eq:grad_meth} with step size $h= \frac{1}{2L}$ the following holds
 \begin{align}\label{eq:function_decreas0}
   f(x^{k+1}) &- f(x^{k})  \leq  -\frac{1}{4L}\|\nabla^r f(x^k, \{\xi^{k+1}_l\}_{l=1}^r)\|^2_2 + \frac{1}{2L}\|\nabla^{r} f(x^k,\{\xi^{k+1}_l\}_{l=1}^{r})-\nabla f(x^k)\|_2^2.
   \end{align}
\end{lemma}
\textit{Proof.}
From Lipschitz continuity of $\nabla f(x)$ we have
 \begin{align}\label{eq:Lipsch}
f(x^{k+1}) &\leq f(x^k) +\la \nabla f(x^k), x^{k+1}-x^k\ra 
+\frac{L}{2}\|x^{k+1}-x^k\|^2_2.
\end{align}
By Cauchy–Schwarz inequality and since $ab \leq \frac{a^2}{2}+\frac{b^2}{2}$ for any $a,b$, we get
\begin{align}\label{eq_fenchel_r}
    \la \nabla f(x^k) - \nabla^r f(x^k,\{\xi^{k+1}_l\}_{l=1}^r), x^{k+1}-x^k \ra &\leq  \frac{1}{2L}\|\nabla f(x^k) - \nabla^r f(x^k,\{\xi^{k+1}_l\}_{l=1}^r)\|^2_2 \notag \\
    &+\frac{L}{2}\|x^{k+1}-x^k\|^2_2.
\end{align}
Then we add and subtract $\nabla^r f(x^k, \{\xi_l^{k+1}\}_{l=1}^r)$  in \eqref{eq:Lipsch}.
Using  \eqref{eq_fenchel_r} 
 we get
 \begin{align}\label{eq:Lipsch_Fench}
     f(x^{k+1}) \leq f(x^k) &+\la \nabla^r f(x^k, \{\xi^{k+1}_l\}_{l=1}^r), x^{k+1}-x^k\ra +\frac{2L}{2}\|x^{k+1}-x^k\|^2_2 +\frac{1}{2L}\delta^2_{k+1},
 \end{align}
where  $\delta^2_{k+1} = \|\nabla^{r} f(x^k,\{\xi^{k+1}_l\}_{l=1}^{r})-\nabla f(x^k)\|_2^2$.
\\
From \eqref{eq:grad_meth}  and \eqref{eq:Lipsch_Fench}  we have
 \begin{align}\label{eq:Lipsch_Fench2}
     f(x^{k+1}) &\leq ~ f(x^k)- h\| \nabla^r f(x^k, \{\xi_l^{k+1}\}_{l=1}^r)\|_2^2  +L\dd{h^2}\|\nabla^r f(x^k, \{\xi_l^{k+1}\}_{l=1}^r)\|^2_2 +\frac{1}{2L}\delta^2_{k+1} \notag\\
      &=f(x^k)  -h(1-L h)\|\nabla^r f(x^k, \{\xi^{k+1}_l\}_{l=1}^r)\|^2_2 +\frac{1}{2L}\delta^2_{k+1}.
 \end{align}
Thus, the step size  $h$ is chosen as follows
 \begin{equation}\label{eq:step_size}
     h = \arg\max_{\alpha\geq 0}\alpha(1-L\alpha) = \frac{1}{2L}.
 \end{equation}
Substituting this $h$ in  \eqref{eq:Lipsch_Fench2} and using definition for $\delta^2_{k+1}$ we finalize the proof.
\qed
 

\begin{algorithm}[ht]
\caption{Stochastic Gradient Descent}
\label{Alg:SGM}

\begin{algorithmic}[1]
        \REQUIRE Number of iterations $N$, variance $D$, Lipschitz constant $L$, accuracy $\eps$, \dd{starting point $x^0$}.
                \STATE Calculate batch size 
        $
            r = \max\{\frac{D}{L\eps}, ~1\}.$
        \FOR{$k = 0,\dots,N-1 $}
 \STATE     \vspace{0.1cm}$
    x^{k+1} = x^k - \frac{1}{2L}\nabla^{r} f(x^k, \{\xi^{k+1}_l\}_{l=1}^{r}).
$\vspace{0.1cm}
                \ENDFOR

        \ENSURE  $\bar x^{N} =\frac{1}{N} \sum_{k=1}^N x^k$.    
\end{algorithmic}
\end{algorithm}

\begin{theorem}\label{Th1}
Algorithm \ref{Alg:SGM} with stochastic gradient oracle calls $ T = O\left(\frac{DR^2}{\eps^2}\right)$, batch size $r =\max\{\frac{D}{L\eps},1\} $, number of iterations $N = O\left(\frac{LR^2}{\eps}\right)$ outputs a point $\bar x^N$ satisfying
\begin{equation}\label{eq:get_goal}
    \E f(\bar x^N) - f(x^*)\leq \eps.
\end{equation}
\end{theorem}

\textit{Proof.} For the sequence $x^1,x^2,...$ generated by Algorithm \ref{Alg:SGM} the following holds
   \begin{align}\label{eq:seq_xk}
       \|x^{k+1} - x\|_2^2 &= \|x^k -x - \frac{1}{2L}\nabla^{r_k} f(x,\{\xi_l^{k+1}\}_{l=1}^r)\|_2^2 \notag \\
       &= \|x^k -x\|_2^2 + \frac{1}{4L^2}\|\nabla^{r_k} f(x^k,\{\xi^{k+1}_l\}_{l=1}^r)\|_2^2 - \frac{1}{L} \la \nabla^{r_k} f(x^k,\{\xi_l^{k+1}\}_{l=1}^r),x^k-x \ra.
   \end{align}


From \eqref{eq:seq_xk} and Lemma \ref{Lm_step_size} we get
\begin{align}\label{eq:final_equ}
    &\la  \nabla^{r} f(x^k,\{\xi^{k+1}_l\}_{l=1}^r), x^k-x \ra \leq f(x^k) - f(x^{k+1}) \notag \\
    &+L\|x^k -x\|_2^2- L\|x^{k+1} -x\|_2^2   +\frac{1}{2L}\delta^2_{k+1},
\end{align}
where $\delta^2_{k+1} = \|\nabla^{r} f(x^k,\{\xi^{k+1}_l\}_{l=1}^{r})-\nabla f(x^k)\|_2^2$.
Using the convexity of $f(x)$, \eqref{eq:stoch_or_asmpt2} and taking the conditional expectation $\E_{x^{k+1}}[~\cdot~| x^1, \dots, x^k]$ from both sides of  \eqref{eq:final_equ}, we have
\begin{align}\label{eq:the_most_final_eq}
    f(x^k) -f(x) &\leq \la \nabla f(x\dd{^k}), x^k - x \ra  \notag \\
    &  \leq f(x^k) - \E_{x^{k+1}}[f(x^{k+1}) | x^1,..., x^k ] + L\|x^k-x\|_2 \notag \\ &- \E_{x^{k+1}}[L\|x^{k+1}-x\|^2_2~ |~ x^1,..., x^k ]+\frac{1}{2L}\E_{x^{k+1}}[\delta^2_{k+1} | x^1,..., x^k ].
\end{align}

Then we summarize this inequality and take the total expectation
\begin{equation}\label{eq:fin_eq_prev}
    \E f(\bar x^N) - f(x^*)\leq \frac{LR^2}{2N} + \frac{1}{2L}\E \delta^2,
\end{equation}
where we used $x=x^*$ and introduced upper bound   $\delta \geq \delta_{k}$ for any $k$.
We choose batch size $r$ in respect with $\E\delta = \eps L$. Since  
    $\E \|\nabla^r f(x^k,\{\xi^{k+1}_l\}_{l=1}^r)-\nabla f(x^k)\|_2^2\leq \frac{D}{r}$
we obtain
\begin{equation}\label{eq:con_r_batch_siez}
r = \max\left\{\frac{D}{L\eps},~1\right\}.
\end{equation}
We define total number of iterations $N$ from \eqref{eq:fin_eq_prev}  such that \eqref{eq:get_goal} holds. Summing $r$ over all iterations we get the total number of oracle calls $T$.
\qed

\subsection{Adaptive algorithm}
Now we assume that the constant $L$ may be unknown (but we can estimate it as $L\in \left[{\underline L}, \bar{L} \right]$ in case when we will obtain theoretical bounds), moreover, if the true variance $D$ is unavailable we use its upper bound $D_0 \geq D$. We provide an adaptive algorithm which iteratively tunes the Lipschitz constant.
Importantly, the approximation of the Lipschitz constant used by the algorithm may decrease as iteration go, leading to larger steps and faster convergence.


Sinse Lipschitz constant $L$ varies from iteration to iteration we need to define different batch size at each iteration. Similar to \eqref{eq:con_r_batch_siez} we choose batch size as follows $r_{k+1} = \max\left\{\frac{D_0}{L_{k+1}\eps}, ~1 \right\}.$
Using \eqref{eq:choose_L} we similarly to \eqref{eq:final_equ} get the following
\begin{align}\label{eq:final_equ_adaptive}
    \la \nabla^{r_{k+1}} f(x,\{\xi^{k+1}_l\}_{l=1}^{r_{k+1}}), x^k-x \ra &\leq f(x^k) - f(x^{k+1}) +L_{k+1}\|x^k -x\|_2^2\notag\\
    &- L_{k+1}\|x^{k+1} -x\|_2^2  + {\eps}/{2}.
\end{align}
Since $L_{k+1}$ is random now, $r_{k+1}$ will be random as well and, consequently, the total number of oracle calls $T$ is not precisely determined. Let us choose it similarly to its counterpart in Theorem \ref{Th1} which ensures \eqref{eq:get_goal}
\begin{equation}\label{eq:adap_total_n_iter}
    T = \sum_{k=1}^{N-1} r_{k+1} =  O\left(\frac{D_0R^2}{\eps^2} \right).
\end{equation}
This number of oracle calls  can be provided by choosing the last batch size $r_N$ as a residual of the sum \eqref{eq:adap_total_n_iter} and calculate the last Lipschitz constant $L_{N} = \frac{D_0}{r_{N}\eps}$. In practice, we do not need to limit ourselves by fixing the number of oracle calls $T$.
\begin{algorithm}[h!]
\caption{Adaptive Stochastic Gradient Descent}
\label{Alg:ASGM}
\begin{algorithmic}[1]
      \REQUIRE Number of iterations $N$, accuracy $\eps$, $\sigma_0$\pd{, initial guess $L_0$}, \dd{starting point $x^0$}. 
        \FOR{$k = 0,\dots,N-1 $}
         \STATE
         $L_{k+1}:=\frac{L_k}{4}$\\
         \REPEAT 
         \STATE $L_{k+1}:=2L_{k+1}$.
        \STATE
        $r_{k+1} = \max\{\frac{D_0}{L_{k+1}\eps}, ~1\}.$
       \STATE $
    x^{k+1} = x^k - \frac{1}{2L_{k+1}}\nabla^{r_{k+1}} f(x^k, \{\xi^{k+1}_l\}_{l=1}^{r_{k+1}}).
$
\UNTIL 
        \begin{align}\label{eq:choose_L}
            f(x^{k+1}) &\leq f(x^k) 
            + \la \nabla^{r_{k+1}} f(x^k,\{\xi_l^{k+1}\}_{l=1}^{r_{k+1}}), x^{k+1} - x^k \ra 
            + L_{k+1}\|x^{k+1} - x^k \|^2_2 + \eps/2.
        \end{align}
                \ENDFOR

        \ENSURE 
        \begin{equation}\label{eq:output:alg2}
                 \bar x^{N}= \frac{\sum_{k=0}^{N-1}\frac{1}{L_{k+1}}x^{k+1}}{\sum_{k=0}^{N-1}\frac{1}{L_{k+1}}}.
        \end{equation}
\end{algorithmic}
\end{algorithm}

From the convexity of $f$ we have
\begin{equation*}
    f(x^k) - f(x) \leq \la\nabla f(x\dd{^k}), x^k - x \ra.  
\end{equation*}
From this it follows 
\begin{align}
    &\la \nabla^{r_k} f(x\dd{^k},\{\xi^{k+1}_l\}_{l=1}^r),  x^k -x \ra \geq  f(x^k) - f(x) +  \la \nabla^{r_k} f(x\dd{^k},\{\xi^{k+1}_l\}_{l=1}^r) - \nabla f(x^k),  x^k -x \ra. 
\end{align}

From this and \eqref{eq:final_equ_adaptive} we have
\begin{align}\label{eq:to_sum}
    &\frac{1}{L_{k+1}}(f(x^k) -f(x)) +  \frac{1}{L_{k+1}}\la \nabla^{r_{k+1}} f(x\dd{^k},\{\xi^{k+1}_l\}_{l=1}^{r_{k+1}}) - \nabla f(x^k),  x^k -x \ra \notag \\
    &\leq\frac{1}{L_{k+1}}\left(f(x^k) - f(x^{k+1})\right) + \|x^k -x\|_2^2  - \|x^{k+1} -x\|_2^2 + \frac{\eps}{2L_{k+1}}.
\end{align}
The same proof steps  with taking conditional expectation, summing progress over iterations and taking the full expectation as for the non-adaptive case fails here. This happens due to this fact: the following sum $\sum_{k=0}^{N-1} \frac{1}{L_{k+1}} \la \nabla^{r_{k+1}} f(x^k,\{\xi_l^{r_{k+1}}\}_{l=1}^{r_{k+1}}) - \nabla f(x^k),  x^k -x \ra$  is not the sum of martingale-differences and therefore the total expectation is not zero, since $r_k$ is random. Thus, unfortunately, we cannot guarantee  that Algorithm \ref{Alg:ASGM} converges in $O\left(\frac{LR^2}{\eps}\right)$ iterations.\footnote{We expect that this problem can be solved by using the technique from \cite{gine2016mathematical,iusem2017extragradient}.}

However,
numerical experiments are in a good agreement with the provided complexity bound. Nevertheless, to overcome theoretical gap, we refer to large deviation technique and slightly change Algorithm \ref{Alg:ASGM}
and modify the assumptions  on the stochastic oracle
\eqref{eq:stoch_or_asmpt2} 
\begin{equation}
     \E \left(\exp(\|\nabla f(x,\xi)-\nabla f(x)\|_2^2 \sigma^{-2})\right) \leq \exp{(1)}.\notag
\end{equation}
If the true variance $\sigma^2$ is unavailable we use its upper bound $\sigma^2_0 \geq \sigma^2$.

\begin{theorem}\label{Th:adaptiv_stoch_grad_descent}
The output $\bar x^N$
of Algorithm \ref{Alg:ASGM} 
with the following change in the step 2: $L_{k+1}:=\max\left\{\frac{L_k}{2}, {\underline  L} \right\}$; in the change step 4 $L_{k+1}:=\min\left\{2L_{k+1},\bar{L}\right\}$; and following change in step 5:
\[r_{k+1} =\max\left\{ \Theta\left(\frac{\sigma^2_0\bar L^2}{L_{k}\eps\underline  L^2}({\ln \alpha^{-1}}+{m\ln N)} )\right),1\right\},\]
    where 
    ${\underline L} \le L_0 \le \bar{L}$,
    ${\underline L} \equiv L_0 \equiv \bar{L} \text{~~mod~~} 2$ and $m={\log_2(\bar L/\underline  L)} \in \mathbb N$,
after $N=\Theta\left(\frac{\bar L R^2}{\eps}\right)$ iterations 
and $T=\tilde O\left(\frac{\sigma_0^2R^2 \bar L^3}{\eps^2 \underline L^3}\right)$ gradient oracle calls,
satisfies the following inequality with probability $\geq 1-\alpha$
\begin{equation}
   f(\bar x^N) -f(x^*)\leq \eps.  \notag
\end{equation}
\end{theorem}
The proof of this theorem is significantly rely on the following lemma.
\begin{lemma}\label{Lm:Delta}
 Let the sequence $x^0,x^1,\dots, x^N$ be generated after $N =O\left(\frac{\bar LR^2}{\eps}\right)$  iterations of Algorithm  \ref{Alg:ASGM} with the change made in Theorem \ref{Th:adaptiv_stoch_grad_descent}. Then   with probability $\geq 1-\alpha$ it holds
\begin{align*}
\sum_{k=0}^{N-1}\frac{1}{L_{k+1}}    \la  \nabla f(x^k) -  &\nabla^{r_{k+1}} f(x^k,\{\xi^{k+1}_l\}_{l=1}^{r_{k+1}}),  ~x^k - x^*  \ra = O\left(\|x^0-x^*\|^2\right)= O\left(R^2\right).
\end{align*}
\end{lemma}
\textit{Sketch of the Proof.}
Let us denote
 \begin{align}
  &\Delta_N (\{L_{k+1}\}_{k=0}^{N-1}) = \sum_{k=0}^{N-1}\frac{1}{L_{k+1}}  \la  \nabla f(x^k) -\nabla^{r_{k+1}} f(x^k,\{\xi^{k+1}_l\}_{l=1}^{r_{k+1}}),  x^k - x^*  \ra. \notag 
\end{align}
Let $L^i_1, L^i_2, \dots L^i_N$ be $i$-th realization of Algorithm \ref{Alg:ASGM}  among all possible realization.  Then from union bound we have
for $t\geq 0$
\begin{align}\label{eq:union_bound_in}
    &\Prob(\Delta_N (\{L_{k+1}\}_{k=0}^{N-1})\geq t)\notag \le \Prob(\cup_{i=1}^{|\{L_k\}|}\Delta_N (\{L^i_{k+1}\}_{k=0}^{N-1})\geq t) \notag \\
    &\leq \sum_{i=1}^{|\{L_k\}|}\Prob(\Delta_N (\{L^i_{k+1}\}_{k=0}^{N-1})\geq t)=|\{L_k\}| \cdot \Prob(\Delta_N (\{L^i_{k+1}\}_{k=0}^{N-1})\geq t),
\end{align}
where $|\{L_k\}|$ is the cardinality (all realizations of $L^i_1, L^i_2, \dots L^i_N$). From steps 2 and 4 of Algorithm \ref{Alg:ASGM}  $|\{L_k\}| \leq N^m$, where we suppose $m={\log_2(\bar L/\underline  L)} \in \mathbb N$.\\
Using Azuma--Hoeffding inequality we have 
\begin{align}\label{eq:Azum_Hoef_in}
    \Prob(\Delta_N (\{L^i_{k+1}\}_{k=0}^{N-1})\geq t) \leq \exp(-t^2/C^2), 
\end{align}
where constant $C$ will be determined further.\\
From \eqref{eq:union_bound_in} and \eqref{eq:Azum_Hoef_in} we have
\begin{align*}
      &\Prob(\Delta_N (\{L_{k+1}\}_{k=0}^{N-1})\geq t)\leq|L_k| \exp{(-t^2/C^2)} =\exp{\left(-t^2/C^2 +m\ln N\right)}.
\end{align*}
 From this inequality with the following change 
\begin{align}
     t=\sqrt{C^2\ln\alpha^{-1} + C^{2}m\ln N}\leq C(\sqrt{\ln\alpha^{-1}} + \sqrt{m\ln N})\notag
\end{align}
we get with probability $\geq 1-\alpha$ the following
\begin{align}\label{eq:last_in}
    \Delta_N (\{L_{k+1}\}_{k=0}^{N-1}) \leq C(\sqrt{\ln\alpha^{-1}} + \sqrt{m\ln N}).
\end{align}
Now we need to estimate $C$.
From Cauchy–Schwarz inequality we have
 \begin{align}
    &\la  \nabla f(x^k) - \nabla^{r_{k+1}} f(x,\{\xi^{k+1}_l\}_{l=1}^{r_{k+1}}),  x^k -x^* \ra \leq \notag\\
    & \max\limits_{k=1,\dots, N}\|x^k -x^*\|_2  \|\nabla f(x^k) - \nabla^{r_{k+1}} f(x^k,\{\xi^{k+1}_l\}_{l=1}^{r_{k+1}})\|_2. \notag
\end{align}

Follow the papers \cite{dvinskikh2019dual,gorbunov2019optimal} we can similarly prove
the following result for the Algorithm \ref{Alg:ASGM}
\begin{align}
\label{eq:R_bound}
\max\limits_{k=1,\dots, N}\|x^k -x^*\|_2 = O\left(\|x^0-x^*\|\dd{_2}\right)=O\left(R\right).
\end{align}
\att{Indeed, denote $R_k = \|x^k - x^*\|$. If we sum \eqref{eq:to_sum} for $k=0 \dots N-1$ for $x = x^*$ and fixed realization $L^i_1, L^i_2, \dots L^i_N$, we obtain
\begin{align}
R_N^2 &\leq R_0^2 + \frac{\eps}{2}\sum_{k=0}^{N-1} \frac{1}{L_{k+1}^i}  + \sum_{k=0}^{N-1} \frac{1}{L_{k+1}^i} \la \nabla f(x^k) - \nabla^{r_{k+1}} f(x\dd{^k},\{\xi^{k+1}_l\}_{l=1}^{r_{k+1}}),  x^k -x^* \ra\\
&\leq R_0^2 + \frac{N \eps}{ 2 \underline L} + \sum_{k=0}^{N-1} \frac{1}{L_{k+1}^i} \la \nabla f(x^k) - \nabla^{r_{k+1}} f(x\dd{^k},\{\xi^{k+1}_l\}_{l=1}^{r_{k+1}}),  x^k -x^* \ra\\
&\leq R_0^2 + O\left(R_0^2 \frac{\bar L}{\underline L}\right) + \sum_{k=0}^{N-1} \frac{1}{L_{k+1}^i} \la \nabla f(x^k) - \nabla^{r_{k+1}} f(x\dd{^k},\{\xi^{k+1}_l\}_{l=1}^{r_{k+1}}),  x^k -x^* \ra.
\end{align}
By induction, let us assume that $R_k^2 \leq A_3 R_0^2$ for all $k = 0 \dots N-1$, we define $A_3$ further. Due to the following inequality
\begin{align}
\frac{1}{L_{k+1}^i} \la \nabla f(x^k) - \nabla^{r_{k+1}} f(x\dd{^k},\{\xi^{k+1}_l\}_{l=1}^{r_{k+1}}),  x^k -x^* \ra \leq \frac{\sqrt{A_3} R_0}{\underline L} \|f(x^k) - \nabla^{r_{k+1}} f(x\dd{^k},\{\xi^{k+1}_l\}_{l=1}^{r_{k+1}})\|_2
\end{align}
for all $k = 0 \dots N-1$ and assumption about stochastic oracle we have, that each conditional variance of each term in the last sum is less or equal to
\begin{align}
O\left(\frac{A_3 R_0^2 \sigma^2}{\underline L^2 r_{k+1}}\right) \leq O\left(\frac{A_3 R_0^2 \eps}{({\ln \alpha^{-1}}+{m\ln N})\underline L }\right).
\end{align}
Using Lemma 2 from \cite{lan2012validation} we have, that with probability $\geq 1 - \alpha$\footnote{We have to insure that with high probability the next inequality holds for all realizations $L_k^i$, thus we have term $\sqrt{{\ln \alpha^{-1}}+{m\ln N}}$.}
\begin{align}
R_N^2 &\leq O\left(R_0^2\right) + O\left(\sqrt{{\ln \alpha^{-1}}+{m\ln N}}\sqrt{\frac{A_3 R_0^2 N \eps}{({\ln \alpha^{-1}}+{m\ln N})\underline L}}\right)\leq A_1 R_0^2 + A_2 \sqrt{A_3} R_0^2,
\end{align}
where $A_1$ and $A_2$ are constants.
 Let us define $A_3$ using equation $A_3 = 1 + A_1 + A_2 \sqrt{A_3}$, thus
\begin{align}
R_N^2 \leq A_1 R_0^2 + A_2 \sqrt{A_3} R_0^2 \leq A_3 R_0^2.
\end{align}
We proved inequality \eqref{eq:R_bound}.
}

Using \eqref{eq:R_bound} we get

\begin{align}\label{eq:to_union}
    &\la  \nabla f(x^k) - \nabla^{r_{k+1}} f(x^k,\{\xi^{k+1}_l\}_{l=1}^{r_{k+1}}),  x^k -x^* \ra \leq cR^2\|\nabla f(x^k) - \nabla^{r_{k+1}} f(x^k,\{\xi^{k+1}_l\}_{l=1}^{r_{k+1}})\|_2, 
\end{align}
where $c>0$.
 Now we estimate $C$ in \eqref{eq:Azum_Hoef_in} using $r_{k+1}$ from Algorithm \ref{Alg:ASGM}, $N=\Theta\left(\frac{\bar LR^2}{\eps}\right)$ and
$ \E \exp\left(\frac{\| \nabla f(x^k) - \nabla^{r_{k+1}} f(x^k,\{\xi^{k+1}_l\}_{l=1}^{r_{k+1}})\|_2^2}{\tilde c \sigma^2/r^{k+1}}\right) \leq \exp{(1)}$, where $\tilde c>1$
 \begin{align}
  &C = \sqrt{3\sum_{k=0}^{N-1}\frac{c^2R^2}{L^2_{k+1}}\frac{\tilde c\sigma}{r_{k+1}}} = \Theta\left(\frac{R^2}{\sqrt{\ln \alpha^{-1}} + \sqrt{m\ln N}} \right).\notag
\end{align}
Substituting  this $C$ in \eqref{eq:last_in} we get the statement of the lemma.
\qed

\textit{Proof of Theorem \ref{Th:adaptiv_stoch_grad_descent}.}\\
To get the total number of oracle calls we summarize the batch size over all iterations
\[
T = \sum_{k=0}^{N-1}\tilde O\left(\frac{\sigma^2_0 \bar L^2}{L_{k}\eps \underline  L^2} \right) \leq \tilde O\left(\frac{\sigma^2_0R^2 \bar L^3}{\eps^2 \underline  L^3} \right).
\]
Let us summarize \eqref{eq:to_sum} over all iterations, divide it by $\sum_{k=0}^{N-1}1/L_{k+1}$ and take $x=x^*$. Then using Jensen's inequality we get
 \begin{align*}
    f(\bar x^N)-f(x^*) &\leq \frac{1}{\sum_{k=0}^{N-1}\frac{1}{L_{k+1}}}\|x^0 -x^*\|_2^2 +\frac{\eps}{2} \\
    & + \frac{1}{\sum_{k=0}^{N-1}\frac{1}{L_{k+1}}}\sum_{k=0}^{N-1}\frac{1}{L_{k+1}} \la  \nabla f(x^k) -\nabla^{r_{k+1}} f(x^k,\{\xi^{k+1}_l\}_{l=1}^{r_{k+1}}),  x^k -x^*  \ra,
\end{align*}
where  $\bar x^N$ is defined in \eqref{eq:output:alg2}
By the definition of
$\bar L$   we have
 \begin{align}\label{eq:last_term_main}
    f&(\bar x^N)-f(x^*)\leq  \frac{\bar LR^2}{N} +\frac{\eps}{2} +\frac{\bar L}{N}\sum_{k=0}^{N-1}\frac{1}{L_{k+1}}   \la  \nabla f(x^k) -\nabla^{r_{k+1}} f(x^k,\{\xi^{k+1}_l\}_{l=1}^{r_{k+1}}),  x^k - x^*  \ra.~ ~ ~ ~ ~ 
\end{align}
Next we use Lemma \ref{Lm:Delta} for the last term in \eqref{eq:last_term_main} and $N = \Theta\left(\frac{\bar L R^2}{\eps}\right)$ to get the statement of the theorem.
\qed

\subsection{Accelerated \pd{adaptive} algorithm}
To compare our complexity bounds for adaptive stochastic gradient descent with the bounds for accelerated variant of our algorithm we refer to \cite{ogaltsov2019heuristic}. For the reader convenience we provide accelerated algorithm  in a simpler form.

\begin{algorithm}[ht]
\caption{Adaptive  Stochastic Accelerated Gradient Method}
\label{Alg:Acc}

\begin{algorithmic}[1]
 \REQUIRE Number of iterations $N$, $D_0$ accuracy $\eps$, $A_0=0$, \pd{initial guess $L_0$,} \dd{$y^0=u^0=x^0$.}
        \FOR{$k = 0,\dots,N-1 $}
        \STATE \pd{$L_{k+1}:=\frac{L_k}{4}.$}
        \REPEAT
        \STATE $L_{k+1} := 2L_{k+1}. $ 
        \STATE $\alpha_{k+1} = ({1+\sqrt{1+\at{4}A_k L_{k+1}}})/{(\at{2}L_{k+1})},$\\
            $A_{k+1} = A_k+\alpha_{k+1}.$
           \STATE 
        $ r_{k+1} = \max\left\{ \frac{\alpha_{k+1}D_0}{\eps}, ~1\right\}$
         \STATE   $ y^{k+1} = (\alpha_{k+1}u^{k} + A_k x^k)/A_{k+1}. $
 \STATE      
$   u^{k+1} = u^k - 
\alpha_{k+1}\nabla^{r_{k+1}} f(y^{k+1}, \{\xi^{k+1}_l\}_{l=1}^{r_{k+1}}).
$
\STATE 
$  x^{k+1} = (\alpha_{k+1}u^{k+1} + A_kx^k)/A_{k+1}.
$
\UNTIL 
        \begin{align}\label{eq:choose_L_non_conv}
            f(x^{k+1}) &\leq f(y^{k+1}) 
            +\la \nabla^{r_{k+1}} f(y^{k+1},\{\xi_l^{k+1}\}_{l=1}^{r_{k+1}}), x^{k+1} - y^{k+1} \ra 
            \notag \\&
            + L_{k+1}\|x^{k+1} - y^{k+1} \|^2_2 + \frac{\alpha_{k+1}}{2A_{k+1}}\eps.
        \end{align}
                \ENDFOR

        \ENSURE     $x^N$. 
\end{algorithmic}
\end{algorithm}

For Algorithm \ref{Alg:Acc} the number of oracle calls $T$ will be the same as for the non-accelerated version of the algorithm (see \eqref{eq:adap_total_n_iter})
while  the number of iterations will be smaller
$N = O\left(\sqrt{{LR^2}/{\eps}}\right).$
Both these bounds are optimal \cite{woodworth2018graph}. 

Unfortunately, to prove these bounds we also met the problem of martingale-differences mentioned above. 
But analogously to Theorem~\ref{Th:adaptiv_stoch_grad_descent} one can obtain a little bit weaker result.

\begin{theorem}\label{Th:fast_adaptiv_stoch_grad_descent}
The output 
of Algorithm \ref{Alg:Acc} 
with the following change in the steps 2: $L_{k+1}:=\max\left\{\frac{L_k}{2}, {\underline  L} \right\}$; and following change in step 4 $L_{k+1}:=\min\left\{2L_{k+1},\bar{L}\right\}$; in step 6:
\[r_{k+1} =\max\left\{ \Theta\left(\frac{\alpha_k\sigma^2_0\bar L^2}{\eps\underline  L^2}({\ln \alpha^{-1}}+{m\ln N)} \right),1\right\};\]
    where 
    ${\underline L} \le L_0 \le \bar{L}$,
    ${\underline L} \equiv L_0 \equiv \bar{L} \text{~~mod~~} 2$ and $m={\log_2(\bar L/\underline  L)} \in \mathbb N$,
after $N=\Theta\left(\sqrt{\frac{\bar L R^2}{\eps}}\right)$ iterations 
and $T=\tilde O\left(\frac{\sigma_0^2R^2 \bar L^3}{\eps^2 \underline L^3}\right)$ gradient oracle calls,
satisfies the following inequality with probability $\geq 1-\alpha$
\begin{equation}
   f(x^N) -f(x^*)\leq \eps.  \notag
\end{equation}
\end{theorem}
We prove the theorem in supplementary materials. 

\subsection{Practical implementation of adaptive algorithms}
Next we comment on applicability of Algorithm \ref{Alg:ASGM} and Algorithm \ref{Alg:Acc} in real problems. Generally, in case when the exact gradients of function $f(x^k)$ is unavailable, 
function values itself of $f(x^k)$ are also unavailable.
It holds, e.g. in stochastic optimization problem, where the objective is presented by its expectation
\begin{equation}\label{eq_expec_func}
    f(x) = \E f(x, \xi).
\end{equation}
In this case we estimate the function as a sample average
\begin{equation}\label{eq:func_value_calc}
    f(x, \{\xi_l\}_{l=1}^r) = \frac{1}{r}\sum_{l=1}^r f(x, \xi_l)
\end{equation}
and use it in adaptive procedures. In this case we interpret $L_k$ as the worst constant among all Lipschitz constants for $f(x, \xi)$ with different realization of $\xi$. Indeed, if $L_{k+1}$ satisfies the following
  \begin{align*}
            f(x^{k+1}, \xi^{k+1}) &\leq f(x^k, \xi^{k+1}) + \la \nabla f(x^k,\xi), x^{k+1} - x^k \ra + L_{k+1}\|x^{k+1} - x^k \|^2_2 + \eps/2.
        \end{align*}
        Then it satisfies 
          \begin{align}\label{eq_adap_proced_repls}
            f(x^{k+1}, \{\xi^{k+1}_l\}_{l=1}^{r_{k+1}}) &\leq f(x^k, \{\xi^{k+1}_l\}_{l=1}^{r_{k+1}})  +\la \nabla^{r_{k+1}} f(x,\{\xi^{k+1}_l\}_{l=1}^{r_{k+1}}), x^{k+1} - x^k \ra \notag \\
            &+ L_{k+1}\|x^{k+1} - x^k \|^2_2 + \eps/2.
        \end{align}
If, e.g, \eqref{eq_expec_func} holds we replace adaptive procedure in the algorithms by \eqref{eq_adap_proced_repls}.

We also comment on batch size. If the batch size $r_k$ decreases during the process of $L_k$ selection, we preserve $r_k$ from the previous iteration in order not to recalculate stochastic approximation $\nabla^{r_{k+1}} f(x\dd{^k}, \{\xi_l^{k+1}\}_{l=1}^{r_{k+1}})$.

All these remarks remain true also in  non-convex case.

\section{Stochastic non-convex optimization}\label{Sect:non_conv}
In this section we assume that   the objective $f$  may be non-convex.  As in the previous section we consider two cases: known and unknown Lipschitz constant $L$.

\subsection{Non-adaptive algorithm}
\begin{algorithm}[ht]
\caption{Non-convex Stochastic Gradient Descent}
\label{Alg:NCSGM}
\begin{algorithmic}[1]
     \REQUIRE Number of iterations $N$, variance $D$, Lipschitz constant $L$, accuracy $\eps$, \dd{starting point $x^0.$}
                \STATE Calculate batch size
$
            r = \max\{\frac{12D}{\eps^2}, ~1\}.
$
        \FOR{$k = 0,\dots,N-1 $}
 \STATE  \vspace{0.1cm}   $
    x^{k+1} = x^k - \frac{1}{2L}\nabla^{r} f(x^k, \{\xi^{k+1}_l\}_{l=1}^{r}).
$\vspace{0.1cm}
                \ENDFOR

        \ENSURE  $ \hat x = \arg\min\limits_{k=1,..N}\|\nabla f(x^k)\|_2$.   
\end{algorithmic}
\end{algorithm}
The next Lemma provides general quite simple inequality which is necessary to prove complexity bounds.

\begin{theorem}\label{Th2}
Algorithm \ref{Alg:NCSGM} with the total number of stochastic gradient oracle calls\footnote{
According to recent works \cite{carmon2017lower,drori2019complexity}, $T$ and $N$ corresponds to lower bounds.} $T = O\left( \frac{\pd{D}L(f(x^0)-f(x^*))}{\eps^4} \right)$
and number of iterations
$N = O\left(\frac{L(f(x^0)-f(x^*))}{\eps^2}\right)$ outputs a point\footnote{ This $\hat x$ is difficult to calculate in practice. Therefore, we refer to the paper \cite{ghadimi2013stochastic},  in which this problem is partially solved.  } $ \hat x^N$ which satisfies
\begin{equation}\label{eq:grad_norm_conv}
     \E\|\nabla f(\hat x^N)\|_2^2 \leq \eps^2.
\end{equation}
\end{theorem}

\textit{Proof.}
Due to  $\|a\|^2 \leq 2\|b\|^2 +2\|a-b\|^2$ for any $a,b\in \R^n$ 
we get the following inequality
\begin{align}\label{eq_lem_func}
\|\nabla^{r} f(x^k, &\{\xi^{k+1}_l\}_{l=1}^{r})\|_2^2 \geq \frac{1}{2}\|\nabla f(x^k)\|_2^2  - \|\nabla f(x^k) - \nabla^{r} f(x^k, \{\xi^{k+1}_l\}_{l=1}^{r})\|_2^2.~ ~ ~
\end{align}
In non-convex case Lemma \ref{Lm_step_size} remains true. Using it and \eqref{eq_lem_func} we get (see also \cite{gasnikov2018power}) 
\begin{equation}\label{eq:function_decreas_non_con}
   f(x^{k+1}) - f(x^{k})  \leq  -\frac{1}{8L}\|\nabla f(x^k)\|^2_2 + \frac{3}{4L}\delta^2_{k+1},
   \end{equation}
   where $\delta^2_{k+1} = \|\nabla^{r} f(x^k,\{\xi^{k+1}_l\}_{l=1}^{r})-\nabla f(x^k)\|_2^2$.
   
   If $\E\delta_{k} \leq \frac{\eps^2}{12} $ for any $k$, then to achieve convergence in the norm of the gradient \eqref{eq:grad_norm_conv}, we need to do $N =16L(f(x^0)-f(x^*))/\eps^2$ iterations.\\
  Then we can define batch size from
   \begin{align}\label{eq_to_batch_size_non_con}
      \E\delta^2_{k+1} &=  \E \|\nabla^{r} f(x^k,\{\xi_l^{k+1}\}_{l=1}^r)-\nabla f(x^k)\|_2^2 =  {D}/{ r} \leq {\eps^2}/{12}. 
   \end{align}
   Consequently, $r = \frac{12D}{\eps^2}$.
Summing $r$ over all $N$ iterations we get the total number of oracle calls $T$.
\qed

\subsection{Adaptive algorithm}
\begin{algorithm}[ht]
\caption{Adaptive Non-convex Stochastic Gradient Descent}
\label{Alg:ANCSGM}
\begin{algorithmic}[1]
 \REQUIRE Number of iterations $N$,  $D_0$, accuracy $\eps$, \pd{initial guess $L_0$,} \dd{starting point $x^0$.}
     \STATE Calculate
        $
            r = \max\{\frac{8D_0}{\eps^2}, ~1\}.
     $
        \FOR{$k = 0,\dots,N-1 $}
        \STATE \pd{$L_{k+1}:=\frac{L_k}{4}$.}
        \REPEAT
        \STATE $L_{k+1} := 2L_{k+1}. $ 
 \STATE   \vspace{0.2cm}
  $  x^{k+1} = x^k - \frac{1}{2L_{k+1}}\nabla^{r} f(x^k, \{\xi^{k+1}_l\}_{l=1}^{r}).$
\UNTIL 
       \begin{align}\label{eq:choose_L_non_conv3}
            f(x^{k+1}) &\leq f(x^k) + \la \nabla^{r} f(x^k,\{\xi_l^{k+1}\}_{l=1}^{r}), x^{k+1} - x^k \ra + L_{k+1}\|x^{k+1} - x^k \|^2_2 + \frac{\eps^2}{32L_{k+1}}.
        \end{align}
                \ENDFOR

        \ENSURE  $ \hat x = \arg\min\limits_{k=1,..N}\|\nabla f(x^k)\|_2$.    
\end{algorithmic}
\end{algorithm}

\begin{theorem}\label{Th3}
Algorithm \ref{Alg:ANCSGM} (with line 3 $L_{k+1}:=\min\left\{\frac{L_k}{4},\underline{L}\right\}$ line 5 $L_{k+1}:=\min\left\{2L_{k+1},2\bar{L}\right\}$) with expected number of stochastic gradient oracle calls   $\tilde T = O\left( \frac{ D_0\bar{L}^2(f(x^0)-f(x^N))}{\underline{L}\eps^4} \right)$ and expected number of iterations $\tilde N = O\left(\frac{ \bar{L}^2(f(x^0)-f(x^N))}{\underline{L}\eps^2} \right)$ outputs a point $\hat x^N$ satisfying
\begin{equation*}
   \E\|\nabla f(\hat x^N)\|_2^2 \leq \eps^2.
\end{equation*}
\end{theorem}   
   
\textit{Sketch of the Proof}. From \eqref{eq:choose_L_non_conv3} using Lemma \ref{Lm_step_size} we get
\begin{align}\label{eq:function_decreas_non_con_adap}
   f(x^{k+1}) - f(x^{k})  \leq  &-\frac{1}{4L_{k+1}}\|\nabla^r f(x^k, \{\xi_l^{k+1}\}_{l=1}^r)\|^2_2+ {\eps^2}/{(32L_{k+1})}.
   \end{align}
From \eqref{eq_lem_func} and \eqref{eq:function_decreas_non_con_adap} and we have
\begin{align}
   f(x^{k+1}) - f(x^{k})  \leq  &-\frac{1}{8L_{k+1}}\|\nabla f(x^k)\|^2_2 + \frac{1}{4L_{k+1}}\delta^2_{k+1} + {\eps^2}/{(32L_{k+1})},
   \end{align}
    where $\delta^2_{k+1} = \|\nabla^{r} f(x^k,\{\xi^{k+1}_l\}_{l=1}^{r})-\nabla f(x^k)\|_2^2$.
    
   If
   $\|\nabla f(x^k)\|^2_2 \geq \eps^2$. Then 
   \begin{equation}\label{eq:function_decreas_non_con_adap2}
   f(x^{k+1}) - f(x^{k})  \leq  - {(3\eps^2 - 8\delta^2_{k+1})}/{(32L_{k+1})}.
   \end{equation}
   We have that $L_{k+1}\leq 2\bar{L}$.  
   
    If $3\eps^2 - 8\delta^2_{k+1}\geq 0$ we may replace $L_{k+1}$ by $2\bar{L}$. Therefore, we rewrite \eqref{eq:function_decreas_non_con_adap2} with minor changes and after taking the expectation we get
     \begin{equation}\label{eq:function_decreas_non_convex_ada}
  \E f(x^{k+1}) - \E f(x^{k})  \leq  - \underline{L}({2\eps^2 - 8\E \delta^2_{k+1}})/{(64\bar{L}^2)}.
   \end{equation}
      Ensuring $\E\delta^2_{k+1} \leq {\eps^2}/{8} $ we obtain
          \begin{equation}\label{eq:func_ad_final}
  \E f(x^{k+1}) - \E f(x^{k})  \leq  - \underline{L} {\eps^2}/{(64\bar{L}^2)}.
   \end{equation}
   Summing this over expected number of iteration  we get
   \begin{equation}\label{eq:n_iter_non_conv}
  \tilde N ={ 64\bar{L}^2(f(x^0)-f(x^*))}/{\underline{L}\eps^2}.
  \end{equation}
  This $  \tilde N$ ensures that for some $k$ we get    $\|\nabla f(x^k)\|^2_2  \pd{\leq} \eps^2$.
  
     We choose the batch size according to
  \begin{align}\label{eq_to_batch_size_non_con2}
      \E\delta^2_{k+1} &=  \E \|\nabla^{r} f(x^k,\{\xi^{k+1}_l\}_{l=1}^{r})-\nabla f(x^k)\|_2^2= {\eps^2}/{8} \leq {D_0}/{r}.
  \end{align} 
      Consequently, $r = \frac{8D_0}{\eps^2}.$
Using the expected number of algorithm iterations \eqref{eq:n_iter_non_conv} we get expected number of oracle calls 
\begin{equation}
   \tilde  T = \tilde Nr ={ 512D_0\bar{L}^2(f(x^0)-f(x^N))}/{\underline{L}\eps^4}.
\end{equation}
\qed

\section{Experiments}\label{Sect:exper}
We perform experiments using proposed methods with and without acceleration\footnote{In practice we use slightly different  rule in line 2: $L_{k+1}:=L_k/4$ and simpler formula for batch size $r_{k+1}$ -- without constants $\underline{L}$ and $\bar{L}$.} on convex and non-convex problems and compare results with commonly used methods --- Adam, \cite{adam} and Adagrad, \cite{duchi2011adaptive}. We trained logistic regression, two-layer sigmoid-activated and ReLU-activated fully-connected networks on MNIST \cite{lecun-mnisthandwrittendigit-2010} and CNN with three filters and three fully-connected layers on CIFAR10 \cite{cifar10}. Objective for all the problems is cross-entropy function between predicted class distribution and ground-truth class. Hyperparameters for Alg. 2, 3 were $D_0 = 0.01, \eps = 10^{-5}, L_k = 100$, and $D_0 = 0.1, L_0 = 1, \eps = 0.002$ for Alg 5. Adam and Adagrad had batch size equals to 128, learning rate $= 0.001$ and $\beta_1 = 0.9, \beta_2 = 0.999$ --- these parameters are frequently used in various machine learning tasks and are used in \cite{adam}. Dynamics of objective function value on training set and testing accuracy for every task are depicted on Fig~\ref{fig:objective}. Since our tasks come from machine learning domain we measure not only objective, but also accuracy on test set Fig~\ref{fig:test}. We also investigate convergence by epochs and sensibility to starting point and hyperparameters on logistic regression and fully connected network. We fix 5 starting points and exponential hyperparameter grids. For our methods the grid was $D_0 = [0.1, 0.01, 0.001, 0.0001]$, $\eps = [0.01, 0.001, 0.0001, 0.00001]$, $L_0 = [1000, 10000]$, min L (minimal cut off for Lipshitz constant for more stable convergence, but it is not necessary) $= [101, 11, 2]$; for Adam and Adagrad we use $lr = [0.00001, 0.0001, 0.001, 0.01, 0.1]$ and batch size $= [32, 64, 128, 256, 512, 1024]$. The procedure is follows. We fix hyperparameters and average all runs by starting point. Then we compute median for each epoch by all hyperparameters (median is used to avoid outliers caused by bad sets of hyperparameters). So, this analysis gives us picture of how algorithms perform in average (by starting points and hyperparameters). Results of the analysis are summarized on Fig~\ref{fig:epochs} and Fig~\ref{fig:epochs_test} for objective and testing accuracy correspondingly. One can see that proposed methods are very robust to hyperparameters set and can be used for wide range of tasks and settings. The code for all proposed methods is available, visit \url{https://github.com/alexo256/Adaptive-Gradient-Descent-for-Convex-and-Non-Convex-Stochastic-Optimization}.

 \begin{figure}[ht]
    \begin{center}
\includegraphics[width=\linewidth]{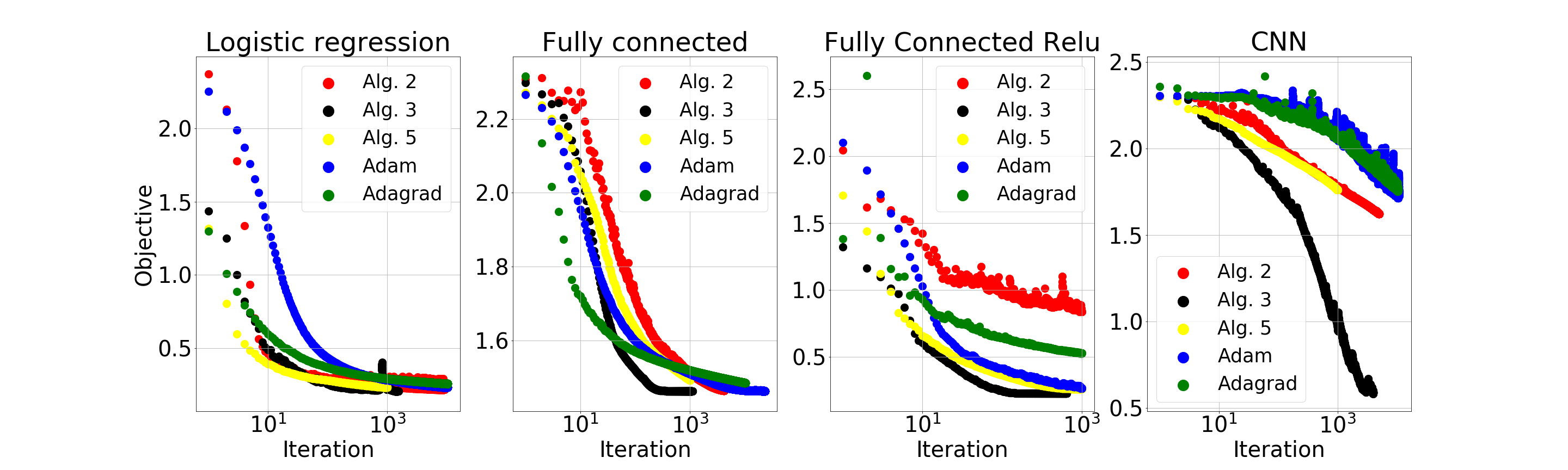}
  \caption{Objective by iteration}
  \label{fig:objective}
    \end{center}
\end{figure}
\begin{figure}[ht]
    \begin{center}
\includegraphics[width=\linewidth]{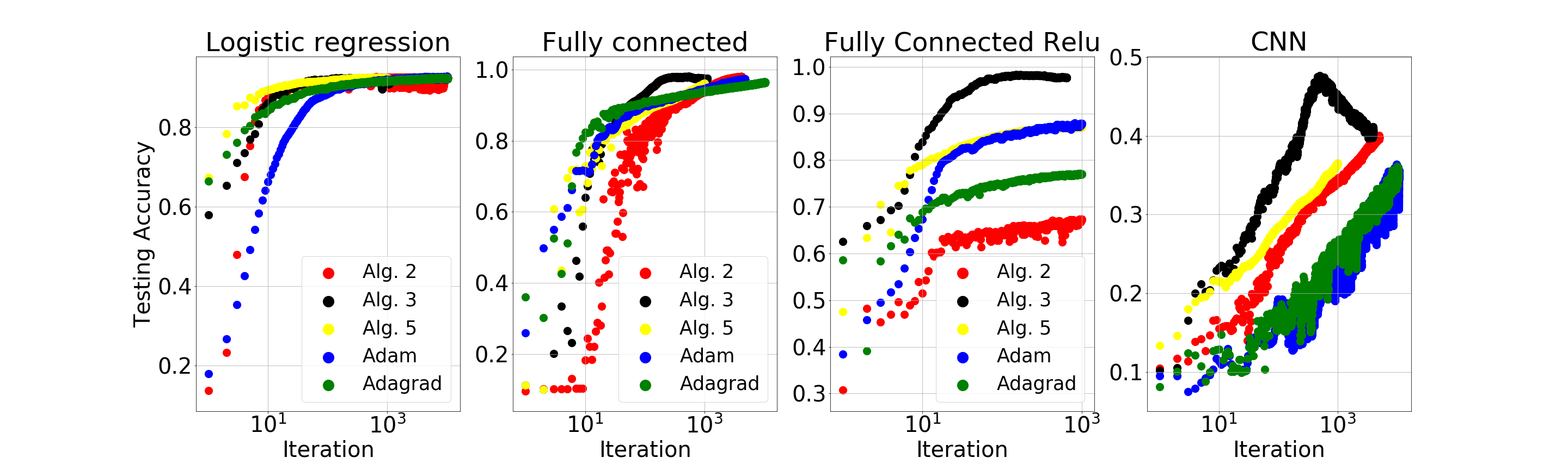}
  \caption{Testing accuracy by iteration}
  \label{fig:test}
    \end{center}
\end{figure}
\begin{figure}[ht]
    \begin{center}
\includegraphics[width=0.9\linewidth]{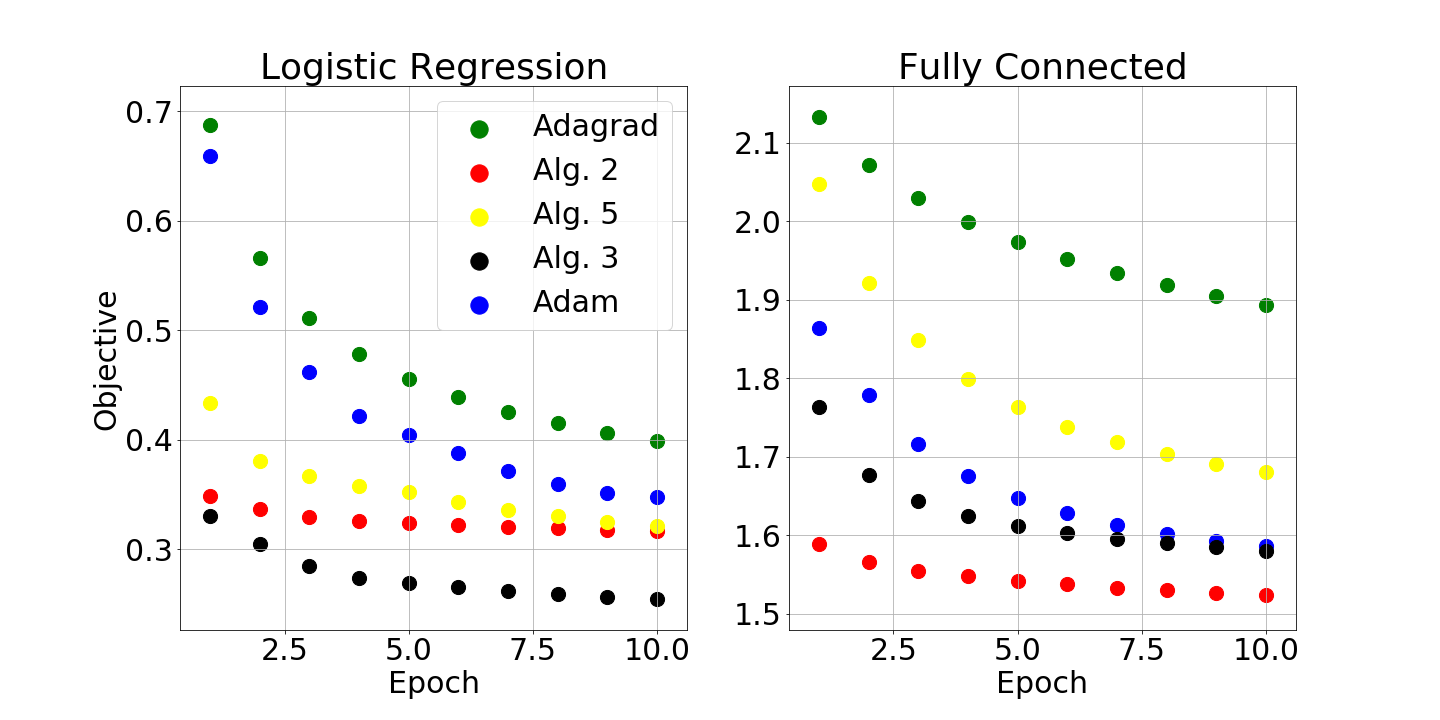}
  \caption{Averaged objective by epoch}
  \label{fig:epochs}
    \end{center}
\end{figure}
\begin{figure}[ht]
    \begin{center}
\includegraphics[width=0.9\linewidth]{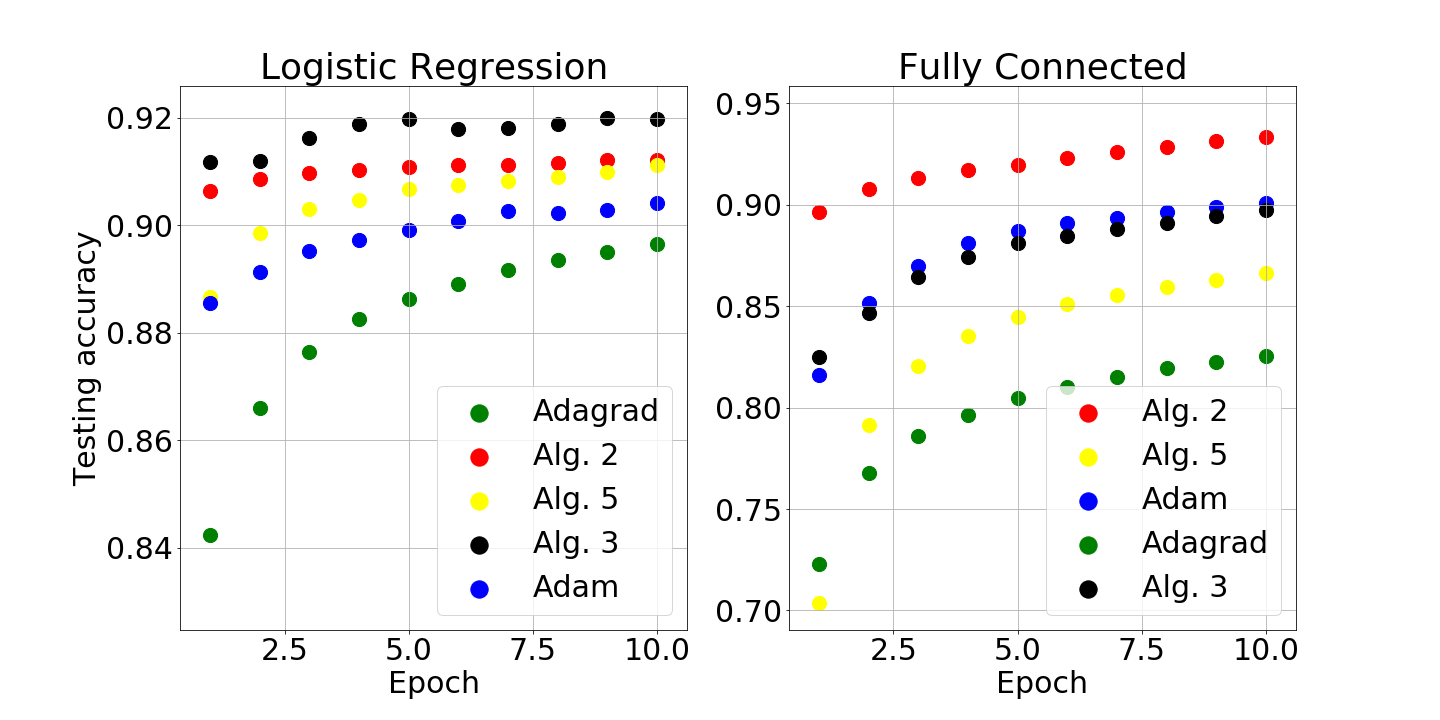}
  \caption{Averaged testing accuracy by epoch}
  \label{fig:epochs_test}
    \end{center}
\end{figure}

\section*{Acknowledgements}
The research of A. Ogaltsov, P. Dvurechensky and V. Spokoiny was partially supported by Huawei. The research of A. Gasnikov was partially supported by Russian Science Foundation project 18-71-00048 mol-a-ved and by Yahoo! Research Faculty Engagement Program. Sect.\ref{Sect:exper} was prepared within the framework of the HSE University Basic Research Program and funded by the Russian Academic Excellence Project `5-100’.

\small
\bibliography{ref}
\bibliographystyle{abbrv}

 \newpage 
 \appendix
 \onecolumn
 \section{Appendix}

\subsection{Accelerated adaptive algorithm}

In order to prove the main result we have to prove the following lemmas.

\begin{lemma}
	Let $\psi(x)$ be a convex function, and 
	\begin{equation*}
	y = {\arg\min_{x \in \R^n}}\{\psi(x) + \frac{\beta}{2} \|x - z\|_2^2\}.
	\end{equation*}
	Then
	\begin{equation*}
	\psi(x) + \frac{\beta}{2} \|x - z\|_2^2 \geq \psi(y) + \frac{\beta}{2} \|y - z\|_2^2 + \frac{\beta}{2} \|x - y\|_2^2 ,\,\,\, \forall x \in \mathbb{R}^n.
	\end{equation*}
	\label{lemma_maxmin_2}
\end{lemma}
The lemma can be proved using optimality condition and $\beta$--strong convexity of the optimized function. Here and after we simplify formula $\nabla^{r} f(\cdot, \{\xi^{k+1}_l\}_{l=1}^{r})$ as $\nabla^{r} f(\cdot)$. Let us denote $l_f(x;y) = f(y) + \langle \nabla^{r_{k+1}} f(y), x - y \rangle$.
\begin{lemma}
	For all $x \in \mathbb{R}^n$
	\begin{align*}
		&\hphantom{{}={}}l_f(x^{k+1};y^{k+1})  + \frac{L_{k+1}}{2}\|{x^{k+1} - y^{k+1}}\|_2^2 \\&\leq \frac{A_k}{A_{k+1}}l_f(x^k;y^{k+1}) +
			 \frac{\alpha_{k+1}}{A_{k+1}}\Bigl(l_f(x;y^{k+1}) 
			 + \frac{1}{2\alpha_{k+1}}\|{x - u^{k}}\|_2^2 - \frac{1}{2\alpha_{k+1}}\|x - u^{k+1}\|_2^2\Bigl)
	\end{align*}
	\label{lemma_maxmin_3DLST}
\end{lemma}
\begin{proof}
	\begin{align*}
	&\hphantom{{}={}}l_f(x^{k+1};y^{k+1})  + \frac{L_{k+1}}{2}\norm{x^{k+1} - y^{k+1}}^2 \\
	&\eqarg{(step 9)}l_f\left(\frac{\alpha_{k+1}u^{k+1} + A_k x^k}{A_{k+1}};y^{k+1}\right)  + \frac{L_{k+1}}{2}\norm{\frac{\alpha_{k+1}u^{k+1} + A_k x^k}{A_{k+1}} - y^{k+1}}^2\\
	&\eqarg{(step 7)} f(y^{k+1}) + \frac{\alpha_{k+1}}{A_{k+1}}\langle \nabla^{r_{k+1}} f(y^{k+1}), u^{k+1} - y^{k+1} \rangle\ \\
	&\hphantom{{}={}}+\frac{A_k}{A_{k+1}}\langle \nabla^{r_{k+1}} f(y^{k+1}), x^k - y^{k+1} \rangle  + \frac{L_{k+1} \alpha^2_{k+1}}{2 A^2_{k+1}}\norm{u^{k+1} - u^k}^2 \\
	 &=\frac{A_k}{A_{k+1}}\left(f(y^{k+1}) + \langle \nabla^{r_{k+1}} f(y^{k+1}), x^k - y^{k+1} \rangle\right)
	 \\&\hphantom{{}={}}+
	 \frac{\alpha_{k+1}}{A_{k+1}}\left(f(y^{k+1}) + 
	 \langle \nabla^{r_{k+1}} f(y^{k+1}), u^{k+1} - y^{k+1} \rangle\right)+\frac{L_{k+1} \alpha^2_{k+1}}{2 A^2_{k+1}}\norm{u^{k+1} - u^k}^2\\ &=_{{\circled{1}}}
	 \frac{A_k}{A_{k+1}}l_f(x^k;y^{k+1})+\frac{\alpha_{k+1}}{A_{k+1}}\left(l_f(u^{k+1};y^{k+1})
	 + \frac{1}{2 \alpha_{k+1}}\norm{u^{k+1} - u^k}^2\right) \\&\leq_{{ \circled{2}}}
	 \frac{A_k}{A_{k+1}}l_f(x^k;y^{k+1}) +
	 \frac{\alpha_{k+1}}{A_{k+1}}\left(l_f(x;y^{k+1})
	 + \frac{1}{2\alpha_{k+1}}\norm{x - u^k}^2 - \frac{1}{2\alpha_{k+1}}\norm{x - u^{k+1}}^2\right).
	\end{align*}

{\small \circled{1}} -- from $A_{k+1} = L_{k+1}\alpha^2_{k+1}$ (step 5).

{\small \circled{2}} -- Lemma \ref{lemma_maxmin_2}. 
Note that we can rewrite step 8 as $$u^{k+1} = {\arg\min_{x \in \mathbb{R}^n}}\{l_f(x;y^{k+1}) + \frac{1}{2\alpha_{k+1}} \norm{x - u^k}^2\}.$$
\end{proof}

\begin{lemma}
	For all $x \in \mathbb{R}^n$,
	\begin{align*}
		&\hphantom{{}={}}A_{k+1} f(x^{k+1}) - A_{k} f(x^{k}) + \frac{1}{2} \norm{x - u^{k+1}}^2 - \frac{1}{2} \norm{x - u^{k}}^2 \\&\leq \alpha_{k+1}f(x) + \frac{\alpha_{k+1}\varepsilon}{2}+ \alpha_{k+1}\langle \nabla^{r_{k+1}} f(y^{k+1})-\nabla f(y^{k+1}), x - u^k \rangle.
	\end{align*}
	\label{lemma_maxmin_3DLST_2}
\end{lemma}
\begin{proof}
\begin{align*}
f(x^{k+1}) &\leqarg{(22) (\text{main part})} l_{f}(x^{k+1};y^{k+1})  + \frac{L_{k+1}}{2}\norm{x^{k+1} - y^{k+1}}^2 + \frac{\alpha_{k+1}}{2A_{k+1}}\varepsilon\\
&\leqarg{Lemma \ref{lemma_maxmin_3DLST}}\frac{A_k}{A_{k+1}}l_{f}(x^k;y^{k+1}) + \frac{\alpha_{k+1}}{A_{k+1}}\Bigl(l_{f}(x;y^{k+1})
	 \\&\hphantom{{}={}}\,\,\,\,\,+ \frac{1}{2\alpha_{k+1}}\norm{x - u^{k}}^2 - \frac{1}{2\alpha_{k+1}}\norm{x - u^{k+1}}^2\Bigl) + \frac{\alpha_{k+1}}{2A_{k+1}}\varepsilon.
\end{align*}
From the last inequality we have
\begin{align*}
 f(x^{k+1}) &\leq \frac{A_k}{A_{k+1}}\left(f(y^{k+1}) + \langle \nabla^{r_{k+1}} f(y^{k+1}), x^k - y^{k+1} \rangle \right) \\&\hphantom{{}={}}+
\frac{\alpha_{k+1}}{A_{k+1}}\Bigl(f(y^{k+1}) + \langle \nabla^{r_{k+1}} f(y^{k+1}), x - y^{k+1} \rangle
	 \\&\hphantom{{}={}}+ \frac{1}{2\alpha_{k+1}}\norm{x - u^{k}}^2 - \frac{1}{2\alpha_{k+1}}\norm{x - u^{k+1}}^2\Bigl) + \frac{\alpha_{k+1}}{2A_{k+1}}\varepsilon \\&= \frac{A_k}{A_{k+1}}\Bigl(f(y^{k+1}) + \langle \nabla f(y^{k+1}), x^k - y^{k+1} \rangle \\&\hphantom{{}={}}+ \langle \nabla^{r_{k+1}} f(y^{k+1})-\nabla f(y^{k+1}), x^k - y^{k+1} \rangle\Bigl) \\&\hphantom{{}={}}+
\frac{\alpha_{k+1}}{A_{k+1}}\Bigl(f(y^{k+1}) + \langle \nabla f(y^{k+1}), x - y^{k+1} \rangle \\&\hphantom{{}={}}+ \langle \nabla^{r_{k+1}} f(y^{k+1})-\nabla f(y^{k+1}), x - y^{k+1} \rangle
	 \\&\hphantom{{}={}}+ \frac{1}{2\alpha_{k+1}}\norm{x - u^{k}}^2 - \frac{1}{2\alpha_{k+1}}\norm{x - u^{k+1}}^2\Bigl) + \frac{\alpha_{k+1}}{2A_{k+1}}\varepsilon \\&\leq_{{\circled{1}}}
	 \frac{A_k}{A_{k+1}}f(x^k) + \frac{\alpha_{k+1}}{A_{k+1}}\Bigl(f(x) + \frac{1}{2\alpha_{k+1}}\norm{x - u^{k}}^2 - \frac{1}{2\alpha_{k+1}}\norm{x - u^{k+1}}^2\Bigl) \\&\hphantom{{}={}}+ \frac{\alpha_{k+1}}{2A_{k+1}}\varepsilon +\frac{\alpha_{k+1}}{A_{k+1}}\Bigl(\langle \nabla^{r_{k+1}} f(y^{k+1})-\nabla f(y^{k+1}), x - y^{k+1} \rangle\Bigl) \\&\hphantom{{}={}}+ \frac{\alpha_{k+1}}{A_{k+1}}\langle \nabla^{r_{k+1}} f(y^{k+1})-\nabla f(y^{k+1}), y^{k+1} - u^k \rangle
	 \\&=
	 	 \frac{A_k}{A_{k+1}}f(x^k) + \frac{\alpha_{k+1}}{A_{k+1}}\left(f(x) + \frac{1}{2\alpha_{k+1}}\norm{x - u^{k}}^2 - \frac{1}{2\alpha_{k+1}}\norm{x - u^{k+1}}^2\right) \\&\hphantom{{}={}}+ \frac{\alpha_{k+1}}{2A_{k+1}}\varepsilon  +\frac{\alpha_{k+1}}{A_{k+1}}\langle \nabla^{r_{k+1}} f(y^{k+1})-\nabla f(y^{k+1}), x - u^{k} \rangle.
\end{align*}

{\small \circled{1}} -- convexity and $A_{k}(y^{k+1} - x^k) = \alpha_{k+1} (u^k - y^{k+1})$ from (step 7).

\end{proof}

\begin{lemma}\label{Lm:Delta_fast}
 Let the sequence $u^0,u^1,\dots, u^N$ and sequence $y^0,y^1,\dots, y^N$ be generated after $N =\Theta\left(\sqrt{\frac{\bar L R^2}{\varepsilon}}\right)$  iterations of Algorithm  3 with the change made in Theorem 3. Then with probability $\geq 1-\alpha$ it holds
\begin{align*}
\sum_{k = 0}^{N-1}\alpha_{k+1}\langle \nabla^{r_{k+1}} f(y^{k+1})-\nabla f(y^{k+1}), x^* - u^k \rangle = O\left(\|x^0-x^*\|^2\right)= O\left(R^2\right).
\end{align*}
\end{lemma}

Proof of Lemma \ref{Lm:Delta_fast} is the same as Lemma 2 (main part). Now we are ready to prove Theorem 3.

\begin{proof}[Proof of Theorem 3]

Let us telescope inequality from Lemma \ref{lemma_maxmin_3DLST_2} for $k = 0, ..., N - 1$, 
	\begin{align*}
	&\hphantom{{}={}}A_{N} f(x^N) - A_{0} f(x^0) + \frac{1}{2}\norm{x - u^N}^2 - \frac{1}{2}\norm{x - u^0}^2 \\&\leq (A_N - A_0)f(x) + \sum_{k = 0}^{N-1}\alpha_{k+1}\langle \nabla^{r_{k+1}} f(y^{k+1})-\nabla f(y^{k+1}), x - u^k \rangle + \frac{A_{N}\varepsilon}{2}
	\end{align*}
	In view of $\frac{1}{2}\norm{x - u^N}^2\geq 0,\,\forall x \in \mathbb{R}^n$, we have
	\begin{align*}
		A_{N} f(x^N) - A_N f(x) &\leq \frac{1}{2}\norm{x - u^0}^2 + \sum_{k = 0}^{N-1}\alpha_{k+1}\langle \nabla^{r_{k+1}} f(y^{k+1})-\nabla f(y^{k+1}), x - u^k \rangle + \frac{A_{N}\varepsilon}{2}.
	\end{align*}
	
	With $x = x^*$ we get
	\begin{align*}
	A_{N} f(x^N) - A_N f(x^*)  &\leq R^2 + \sum_{k = 0}^{N-1}\alpha_{k+1}\langle \nabla^{r_{k+1}} f(y^{k+1})-\nabla f(y^{k+1}), x - u^k \rangle + \frac{A_{N}\varepsilon}{2}.
	\end{align*}
Using Lemma \ref{Lm:Delta_fast} we have that with probability $\geq 1-\alpha$ it holds
	\begin{align*}
	A_{N} f(x^N) - A_N f(x^*)  &\leq R^2 + O\left(R^2\right) + \frac{A_{N}\varepsilon}{2}
	\end{align*}
	and
	\begin{align*}
	f(x^N) - f(x^*)  &\leq \frac{R^2}{A_N} + O\left(\frac{R^2}{A_N}\right) + \frac{\varepsilon}{2}.
	\end{align*}
Next we use $\frac{1}{A_N} = O\left(\frac{\bar L}{N^2}\right)$ and $N = \Theta\left(\sqrt{\frac{\bar L R^2}{\varepsilon}}\right)$ in order to get the statement of the theorem. It is remains to show that
$T=\tilde O\left(\frac{\sigma_0^2R^2 \bar L^3}{\varepsilon^2 \underline L^3}\right)$. Indeed, from $A_N = O\left(\frac{N^2}{\underline L}\right)$ we have
\[
T = \sum_{k=0}^{N-1}\tilde O\left(\frac{\alpha_k\sigma^2_0\bar L^2}{\varepsilon\underline  L^2}\right) = \tilde O\left(\frac{A_N \sigma^2_0\bar L^2}{\varepsilon\underline  L^2}\right) = \tilde O\left(\frac{\sigma^2_0 R^2\bar L^3}{\varepsilon^2\underline  L^3}\right).
\]
\end{proof}

\end{document}